\documentclass[12pt,reqno]{amsart}

\usepackage{amsmath, amsthm, amssymb,verbatim,fullpage,longtable,multicol,multirow}
\usepackage{graphicx}
\usepackage{hyperref}
\usepackage[table]{xcolor}
\usepackage[T1]{fontenc}
\usepackage{appendix}
\usepackage[utf8]{inputenc}
\usepackage{tikz}
\usepackage{braket}
\usepackage{mathtools}
  \usepackage{subfig}
        \usepackage{multicol}
        \usepackage{wrapfig}
        \usepackage{float}
\usepackage{color}

\usepackage{enumitem}
\usepackage{caption}
\usepackage{hyperref}
\newtheorem{theorem}{Theorem}[section]
\newtheorem{lemma}[theorem]{Lemma}
\newtheorem{proposition}[theorem]{Proposition}

\newcommand{\cm}{c_{{\rm mean}}}
\newcommand{\cv}{c_{{\rm var}}}
\newcommand{\E}{\mathbb{E}}
\definecolor{lightgray}{gray}{0.7}
\definecolor{midgray}{gray}{.9}
\newcommand\be{\begin{equation}}
\newcommand\ee{\end{equation}}
\newcommand\bea{\begin{eqnarray}}
\newcommand\eea{\end{eqnarray}}
\newcommand\bi{\begin{itemize}}
\newcommand\ei{\end{itemize}}
\newcommand\ben{\begin{enumerate}}
\newcommand\een{\end{enumerate}}
\DeclareMathOperator{\Var}{Var}

\numberwithin{equation}{section}

\begin{document}

\title{Individual Gap Measures from Generalized Zeckendorf Decompositions}

\author{Robert Dorward}
\email{\textcolor{blue}{\href{mailto:rdorward@oberlin.edu}{rdorward@oberlin.edu}}}
\address{Dept. of Mathematics, Oberlin College, Oberlin, OH 44074}

\author{Pari L. Ford}
\email{\textcolor{blue}{\href{mailto:fordpl@bethanylb.edu}{fordpl@bethanylb.edu}}}
\address{Dept. of Mathematics and Physics, Bethany College, Lindsborg, KS 67456}

\author{Eva Fourakis}
\email{\textcolor{blue}{\href{mailto:erf1@williams.edu}{erf1@williams.edu}}}
\address{Dept. of Mathematics and Statistics, Williams College, Williamstown, MA 01267}

\author{Pamela E. Harris}
\email{\textcolor{blue}{\href{mailto:pamela.harris@usma.edu}{pamela.harris@usma.edu}}}
\address{Dept. of Mathematical Sciences, United States Military Academy, West Point, NY 10996}

\author{Steven J. Miller}
\email{\textcolor{blue}{\href{mailto:sjm1@williams.edu, Steven.Miller.MC.96@aya.yale.edu}{sjm1@williams.edu, Steven.Miller.MC.96@aya.yale.edu}}}
\address{Dept. of Mathematics and Statistics, Williams College, Williamstown, MA 01267}

\author{Eyvindur A. Palsson}
\email{\textcolor{blue}{\href{mailto:eap2@williams.edu}{eap2@williams.edu}}}
\address{Dept. of Mathematics and Statistics, Williams College, Williamstown, MA 01267}

\author{Hannah Paugh}
\email{\textcolor{blue}{\href{mailto:hannah.paugh@usma.edu}{hannah.paugh@usma.edu}}}
\address{Dept. of Mathematical Sciences, United States Military Academy, West Point, NY 10996 }

\date{\today}



\thanks{The authors thank the AIM REUF program, the SMALL REU and Williams College, and West Point for generous  support. The first and third named authors were supported by NSF Grant DMS1347804, and the fifth named author by NSF Grant DMS1265673. This research was performed while the fourth named author held a National Research Council Research Associateship Award at USMA/ARL}

%\author[3]{Robert Dorward\thanks{rdorward@oberlin.edu.}}
%\author[2]{Pari Ford\thanks{fordpl@unk.edu.}}
%\author[4]{Eva Fourakis\thanks{erf1@williams.edu.}}
%\author[1]{Pamela E. Harris\thanks{pamela.harris@usma.edu. This research was performed while the author held a National Research Council Research Associateship Award at USMA/ARL.}}
%\author[4]{Steven J. Miller\thanks{\fix{sjm1@williams.edu, Steven.Miller.MC.96@aya.yale.edu.}}}
%\author[4]{Eyvindur Palsson\thanks{\fix{Eyvindur.A.Palsson@williams.edu}}}
%\author[1]{Hannah Paugh\thanks{hannah.paugh@usma.edu.}}
%\affil[1]{Department of Mathematical Sciences, United States Military Academy}
%\affil[2]{Department of Mathematics, University of Nebraska at Kearney}
%\affil[3]{Department of Mathematics, Oberlin College}
%\affil[4]{Department of Mathematics \& Statistics, Williams College}
%\renewcommand\Authands{ and }

\maketitle

\begin{abstract}
Zeckendorf's theorem states that every positive integer can be uniquely decomposed as a sum of nonconsecutive Fibonacci numbers. The distribution of the number of summands converges to a Gaussian, and the individual measures on gaps between summands for $m \in [F_n, F_{n+1})$ converge to geometric decay for almost all $m$ as $n\to\infty$. While similar results are known for many other recurrences, previous work focused on proving Gaussianity for the number of summands or the average gap measure. We derive general conditions which are easily checked yield geometric decay in the individual gap measures of generalized Zeckendorf decompositions attached to many linear recurrence relations.
\end{abstract}

%MSC Codes: 	05E10
%%%%%%%%%%%%%%%%%%%%%%%%%%%%%%%%%%%%%%%%%%%%%%%%%%%%%%%%%%%%%%%%%%%%%%%%%%%%%%%%%%%%%%%%%%%%%%%%%%%%%%%%%%%%%%%%%%%%%%%%%%%%%%%%%%%%%%%%%%%%%%%%%%%%
%%%%%%%%%%%%%%%%%%%%%%%%%%%%%%%%%%%%%%%%%%%%%%%%%%%%%%%%%%%%%%%%%%%%%%%%%%%%%%%%%%%%%%%%%%%%%%%%%%%%%%%%%%%%%%%%%%%%%%%%%%%%%%%%%%%%%%%%%%%%%%%%%%%%
%%%%%%%%%%%%%%%%%%%%%%%%%%%%%%%%%%%%%%%%%%%%%%%%%%%%%%%%%%%%%%%%%%%%%%%%%%%%%%%%%%%%%%%%%%%%%%%%%%%%%%%%%%%%%%%%%%%%%%%%%%%%%%%%%%%%%%%%%%%%%%%%%%%%
%%%%%%%%%%%%%%%%%%%%%%%%%%%%%%%%%%%%%%%%%%%%%%%%%%%%%%%%%%%%%%%%%%%%%%%%%%%%%%%%%%%%%%%%%%%%%%%%%%%%%%%%%%%%%%%%%%%%%%%%%%%%%%%%%%%%%%%%%%%%%%%%%%%%
\section{Introduction} \label{sec:intro}

If we define the Fibonaccis by $F_1 = 1$, $F_2 = 2$ and $F_{n+1} = F_n + F_{n-1}$, Zeckendorf \cite{Ze} proved the remarkable property that every positive integer can be uniquely written as a sum of non-consecutive Fibonacci numbers (this property is equivalent to the definition of the Fibonaccis). Zeckendorf's theorem has been generalized to other sequences, see among others \cite{Al,Day,DDKMMV,DDKMV,DG,GT,GTNP,Ste1,Ste2}. Many authors proved that sequences $\{a_n\}$ defined by suitable linear recurrences lead to unique decompositions, with the number of summands of $m \in [a_n, a_{n+1})$ converging to a Gaussian (see for example \cite{LT,MW}) and the average gap measure converging to geometric decay (see \cite{BBGILMT, BILMT}). It is significantly easier to focus on the average gap measures rather than the individual gap measures associated to each $m$; in this note we isolate a general set of conditions which suffice to prove these individual measures converge almost surely to geometric decay.

We work in great generality so the arguments below will apply to numerous sequences. We assume we have a sequence $\{b_n\}$ and a decomposition rule that leads to unique decomposition. Fix constants $c_1,d_1, c_2,d_2$ such that $I_n:= [b_{c_1n+d_1},b_{c_2n+d_2})$ is a well-defined interval for all $n>0$. Let $\delta(x-a)$ denote the Dirac delta functional (assigning a mass of 1 to $x=a$ and 0 otherwise), $k(z)$ be the number of summands in $z$'s decomposition ($z = b_{\ell_1} + \cdots + b_{\ell_{k(z)}})$, and the total number of gaps for all $z\in I_n$ is
\begin{equation}
N_{\rm gaps}(n) \ :=\ \sum_{z=b_{c_1n+d_1}}^{b_{c_2n+d_2}-1}(k(z)-1).
\end{equation}

\begin{itemize}

\item \emph{Spacing gap measure:} We define the spacing gap measure of a $z \in I_n$ by
\begin{equation}
\nu_{z,n}(x) \ := \ \frac{1}{k(z)-1}\sum_{j=2}^{k(z)}\delta(x-(\ell_j-\ell_{j-1})).
\end{equation}
% Note that $\ell_j = \ell_j(z)$, though we omit this dependence on $z$ as it should be clear from context and reduces clutter from the notation.

\item \emph{Average spacing gap measure:} 
The average spacing gap measure for all $z\in I_n$ is
\begin{align} \nu_n(x)  := \frac1{N_{{\rm gaps}}(n)} \sum_{z=b_{c_1n+d_1}}^{b_{c_2n+d_2}-1} \sum_{j=2}^{k(z)} \delta\left(x - (\ell_j - \ell_{j-1})\right)  =  \frac1{N_{{\rm gaps}}(n)} \sum_{z=b_{c_1n+d_1}}^{b_{c_2n+d_2}-1} \left(k(z)-1\right) \nu_{z,n}(x).
\end{align}
Letting $P_n(g)$ denote the probability of a gap of length $g$ among all gaps from the decompositions of all $m\in I_n$, we have
\begin{equation}
 \nu_n(x) \ = \ \sum_{g=0}^{c_2n+d_2-1} P_n(g) \delta(x - g).
\end{equation}

\item \emph{Limiting average spacing gap measure, limiting gap probabilities:} If the limits exist:
\begin{equation} \label{gaplim}
\nu(x) \ := \ \lim_{n\to\infty} \nu_n(x), \ \ \ \ P(k) \ := \ \lim_{n \to \infty} P_n(k).
\end{equation}

%\item \emph{Indicator function for one gap:}
%For $g \ge 0$ let
%\begin{equation}
% X_{i,i+g}(n)\ =\  \#\{z \in I_n:\ \text{$b_i$, $b_{i+g}$ in $z$'s decomposition, but not $G_{i+q}$ for $0 < q < g$}\}.
%\end{equation}

\item \emph{Indicator function for two gaps:}
For $g_1,g_2 \ge 0$
\begin{align}
X_{j_1,j_1+g_1, j_2,j_2+g_2}(n)  & \ := \  \#\left \{
z \in I_n:\begin{subarray}\ b_{j_1}, b_{j_1+g_1}, b_{j_2}, b_{j_2+g_2}\ \text{in}\ z\text{'s\ decomposition,}\\
\text{but\ not\ } b_{j_1+q}, b_{j_2+p}\ \text{for}\ 0<q<g_1, 0<p<g_2
\end{subarray}
\right \}.
\end{align}

%\item \emph{Specific gap length probability:} Recall that $P_n(g)$ is the probability
%\begin{equation}
%P_n(g)\ :=\ \frac{1}{N_{\rm gaps}(n)}\sum_{i=1}^{c_2n+d_2-g}X_{i,i+g}(n).
%\end{equation}

\end{itemize}

We generalize the work in \cite{BILMT}. The authors there concentrated on a specific class of recurrences; our arguments are general. In addition to holding for the oft studied positive linear recurrences, they hold for new systems such as the $m$-gonal numbers of \cite{DFFHMPP}, as well as sequences without unique decomposition \cite{CFHMN}.

\begin{theorem}\label{genthm} For $z\in I_n$, the individual gap measures $\nu_{z,n}(x)$ converge almost surely in distribution  to the average gap measure $\nu(x)$ if the following hold.

\begin{enumerate}
\item The number of summands for decompositions of $z \in I_n$ converges to a Gaussian with mean $\mu_n = \cm n+O(1)$ and variance $\sigma_n^2 = \cv n+O(1)$, for constants $\cm, \cv>0$, and $k(z) \ll n$ for all $z \in I_n$.

\item We have the following, with $\lim_{n\to\infty} \sum_{g_1, g_2} {\rm error}(n,g_1,g_2) = 0$:
\begin{align}
\frac{2}{|I_n|\mu_n^2}\sum_{j_1<j_2}X_{j_1,j_1+g_1,j_2,j_2+g_2}(n) \ = \  P(g_1)P(g_2)+\text{{\rm error}}(n,g_1,g_2).
\end{align}

%\item For all $z\in I_n$, $k(z)\ll n$.

\item The limits in Equation~(\ref{gaplim}) exist.
\end{enumerate}

\end{theorem}

%%%%%%%%%%%%%%%%%%%%%%%%%%%%%%%%%%%%%%%%%%%%%%%%%%%%%%%%%%%%%%%%%%%%%%%%%%%%%%%%%%%%%%%%%%%%%%%%%%%%%%%%%%%%%%%%%%%%%%%%%%%%%%%%%%%%%%%%%%%%%%%%%%%%
%%%%%%%%%%%%%%%%%%%%%%%%%%%%%%%%%%%%%%%%%%%%%%%%%%%%%%%%%%%%%%%%%%%%%%%%%%%%%%%%%%%%%%%%%%%%%%%%%%%%%%%%%%%%%%%%%%%%%%%%%%%%%%%%%%%%%%%%%%%%%%%%%%%%
%%%%%%%%%%%%%%%%%%%%%%%%%%%%%%%%%%%%%%%%%%%%%%%%%%%%%%%%%%%%%%%%%%%%%%%%%%%%%%%%%%%%%%%%%%%%%%%%%%%%%%%%%%%%%%%%%%%%%%%%%%%%%%%%%%%%%%%%%%%%%%%%%%%%
%%%%%%%%%%%%%%%%%%%%%%%%%%%%%%%%%%%%%%%%%%%%%%%%%%%%%%%%%%%%%%%%%%%%%%%%%%%%%%%%%%%%%%%%%%%%%%%%%%%%%%%%%%%%%%%%%%%%%%%%%%%%%%%%%%%%%%%%%%%%%%%%%%%%

\section{Proof of Theorem \ref{genthm}}\label{App:B}

We need the following definitions.

\begin{itemize}
\item $\widehat {\nu_{z,n}}(t)$: The characteristic function of $\nu_{z,n}(x)$.

\item $\widehat \nu(t)$: The characteristic function of the limiting average gap distribution $\nu(x)$.

\item $\E_z[\dots]$: The expected value over $z \in I_n$ with the uniform measure:
\begin{equation}
\E_z[X] \ :=\ \frac{1}{|I_n|}\sum_{z=b_{c_1n+d_1}}^{b_{c_2n+d_2}-1}X(z).
\end{equation}
\item \emph{Indicator function for one gap:}
For $g \ge 0$ let
\begin{equation}
 X_{i,i+g}(n)\ =\  \#\{z \in I_n:\ \text{$b_i$, $b_{i+g}$ in $z$'s decomposition, but not $G_{i+q}$ for $0 < q < g$}\}.
\end{equation}

\end{itemize}
\begin{proposition} \label{mean}
We have
\begin{align}
\lim_{n\rightarrow\infty}\E_z[\widehat{\nu_{z;n}}(t)]\ = \ \widehat{\nu}(t).
\end{align}
\end{proposition}
First notice that
\begin{align}
\widehat{\nu_{z,n}}(t) \ := \   \int_{0}^{\infty} e^{ixt}\nu_{z,n}(t)dx \ = \  \frac{1}{k(z)-1}\sum_{j=2}^{k(z)}e^{it(\ell_j-\ell_{j-1})},
\end{align}
where $z= b_{\ell_1}+\dots+b_{\ell_{k(z)}}.$
Therefore
\begin{equation}
\E_z[\widehat{\nu_{z,n}}(t)] \ = \  \frac{1}{|I_n|} \sum_{z = b_{c_1n+d_1}}^{b_{c_2n+d_2}-1} \frac{1}{k(z)-1}\sum_{j=2}^{k(z)}e^{it(\ell_j-\ell_{j-1})}.
\end{equation}

\begin{lemma} \label{lemmamean1}
We have
\begin{equation}\lim_{n\rightarrow\infty} \frac{1}{|I_n|} \sum_{z = b_{c_1n+d_1}}^{b_{c_2n+d_2}-1}\left(\frac{(k(z)-1)-\mu_n}{(k(z)-1)\mu_n}\right)\sum_{j=2}^{k(z)}e^{it(\ell_j-\ell_{j-1})}\ =\ 0.\end{equation}
\end{lemma}

\begin{proof}
We break into cases based on how far away $k(z)$ is from the mean. For $0<\delta<1/2$ 
\begin{align}
I_n(\delta)\ : = \ \{z\in I_n: k(z) \in [\mu_n- (\cv n)^{1/2}, \mu_n+(\cv n)^{1/2}]\}
\end{align}
\textbf{Case 1}: Let $z\in I_n(\delta)$. Thus $k(z)$ is close to $\mu_n$. As $k(z) \ll n$
\begin{eqnarray}
\frac{1}{|I_n|} \sum_{\substack{z = b_{c_1n+d_1}\\z\in I_n(\delta)}}^{b_{c_2n+d_2}-1}\left(\frac{(k(z)-1)-\mu_n}{(k(z)-1)\mu_n}\right)\sum_{j=2}^{k(z)}e^{it(\ell_j-\ell_{j-1})} & \ \ll \ & \frac{1}{|I_n|} \sum_{z = b_{c_1n+d_1}}^{b_{c_2n+d_2}-1}\frac{n^{1/2+\delta}}{n^2}\sum_{j=2}^{k(z)}e^{it(\ell_j-\ell_{j-1})}\nonumber\\
& \ \ll \ &\frac{|I_n|n}{n^{3/2-\delta}|I_n|} \ = \  n^{-1/2+\delta},
\end{eqnarray}
where the last line follows because $k(z)\ll n$. \\
\textbf{Case 2}: Let $k(z)\notin I_n(\delta)$.
By Gaussianity, for sufficiently large $n$, the probability that $z\in I_n$ is in this case is essentially
\begin{align}
2\int_{\cv^{1/2} n^{1/2+\delta}}e^{-t^2/2b^2n}dt\ll e^{-n^{2\delta}/2}.
\end{align}
Therefore the number of integers $z\in I_n\setminus I_n(\delta)$ is essentially $|I_n|e^{-n^{2\delta}/2}.$ Thus
\begin{align}
\frac{1}{|I_n|} \sum_{\substack{z = b_{c_1n+d_1}\\z\notin I_n(\delta)}}^{b_{c_2n+d_2}-1}\left(\frac{(k(z)-1)-\mu_n}{(k(z)-1)\mu_n}\right)\sum_{j=2}^{k(z)}e^{it(\ell_j-\ell_{j-1})} \ \ll \ \frac{1}{|I_n|} \cdot |I_n|e^{-n^{2\delta}/2}\cdot n\  = \ ne^{-n^{2\delta}/2},
\end{align}
which tends to zero as $n\rightarrow \infty$ and proves the claim.
\end{proof}
Through a similar argument we have
\begin{lemma}\label{lemmamean2}
\begin{align}
\lim_{n\rightarrow\infty} \frac{1}{|I_n|} \sum_{z = b_{c_1n+d_1}}^{b_{c_2n+d_2}-1}\left(\frac{(k(z)-1)^2-\mu_n^2}{(k(z)-1)^2\mu_n^2}\right)\left(\sum_{j=2}^{k(z)}e^{it(\ell_j-\ell_{j-1})}\right)^2 = 0.
\end{align}
\end{lemma}

Proposition~\ref{mean} now follows.
\begin{proof}[Proof of Proposition~\ref{mean}]
By Lemma~\ref{lemmamean1}, we replace $\frac{1}{k(z)-1}$ with $\frac{1}{\mu_n}$ with negligible error:
\begin{align}
\E_z[\widehat{\nu_{z,n}}(t)] &\ = \  \frac{1}{|I_n|} \sum_{z = b_{c_1n+d_1}}^{b_{c_2n+d_2}-1} \frac{1}{k(z)-1}\sum_{j=2}^{k(z)}e^{it(\ell_j-\ell_{j-1})} \ = \  \frac{1}{|I_n|\mu_n}  \sum_{z = b_{c_1n+d_1}}^{b_{c_2n+d_2}-1}\sum_{j=2}^{k(z)} e^{it(\ell_j-\ell_{j-1})} + o(1)\nonumber\\
&\ = \ \frac{1}{|I_n|\mu_n}  \sum_{g = 0}^{c_2n+d_2-1} \sum_{j=1}^{c_2n+d_2-g}X_{j,j+g}(n)e^{itg}+ o(1) \ = \ \sum_{g=0}^{c_2n+d_2-1}P_n(g)e^{itg}+ o(1),
\end{align}
with the last equality follows by definition. Then
\begin{equation}
\lim_{n\rightarrow\infty} \E_z[\widehat{\nu_{z,n}}(t)]\ = \ \lim_{n\rightarrow\infty} \Bigg(\sum_{g=0}^{c_2n+d_2-1}P_n(g)e^{itg}+ o(1)\Bigg)\ = \ \sum_{g=0}^{\infty}P(g)e^{itg}\ = \ \widehat{\nu}(t),
\end{equation}
which completes the proof.
\end{proof}

\begin{proposition}\label{var}
We have
\begin{equation}
\lim_{n\rightarrow \infty} \Var_n(t) \ :=\ \lim_{n\rightarrow \infty} \E_z[(\widehat{\nu_{z,n}}(t)-\widehat{\nu}_n(t))^2] \ = \  0.
\end{equation}
\end{proposition}
\begin{proof}
Note that
\begin{equation}
\Var_n(t) \ :=\ \lim_{n\rightarrow \infty} \E_z[(\widehat{\nu_{z,n}}(t)-\widehat{\nu}_n(t))^2]\ = \ \E_z[\widehat{\nu_{z,n}}(t)^2]-\widehat{\nu}_n(t)^2.
\end{equation}
We show that $\lim_{n\rightarrow\infty}\E_z[\widehat{\nu_{z,n}}(t)^2]$ differs from \begin{equation}
\widehat{\nu}(t)^2\ = \ \sum_{g_1=0}^{\infty}P(g_1)e^{itg_1}\sum_{g_2=0}^{\infty}P(g_2)e^{itg_2}
\ = \ \sum_{g_1,g_2}P(g_1)P(g_2)e^{it(g_1+g_2)}
\end{equation}
by $o(1)$. Let $g_1$ and $g_2$ be two arbitrary gaps starting at the indices $j_1\leq j_2$. We have
\begin{eqnarray}
& & \E_z[\widehat{\nu_{z,n}}(t)^2] \ = \  \frac{1}{|I_n|} \sum_{z=b_{c_1n+d_1}}^{b_{c_2n+d_2}-1}\frac{1}{(k(z)-1))^2}\sum_{r=2}^{k(z)}e^{it(\ell_r(z)-\ell_{r-1}(z))}\sum_{w=2}^{k(z)}e^{it(\ell_w(z)-\ell_{w-1}(z))}\nonumber\\
& &\ \ \ = \ \frac{1}{|I_n|}\frac{1}{\mu_n}\left(2\sum_{\substack{j_1<j_2\\ g_1,g_2}}X_{j_1,j_1+g_1,j_2,j_2+g_2}(n)e^{itg_1}e^{itg_2} +\sum_{j_1,g_1}X_{j_1,j_1+g_1}(n)e^{2itg_1}\right)+o(1).\ \ \ \ \ 
\end{eqnarray}
The last line follows by Lemma~\ref{lemmamean2} (the 2 is from $j_1<j_2$). The diagonal term doesn't contribute to the limit as the denominator is of order $n^2|I_n|$ and $\sum_{j_1,g_1}X_{j_1,j_1+g_1}(n)e^{2itg_1}$ is of order $n|I_n|$. Using the second condition from Theorem~\ref{genthm} gives $\lim_{n\rightarrow\infty}\Var_n(t)=0$.
\end{proof}

\begin{proof}[Proof of Theorem~\ref{genthm}] L\'{e}vy's criterion (see \cite{FG}) states that if a sequence of random variables $\{R_n\}$ whose characteristic functions $\{\phi_n\}$ converge pointwise to $\phi$, where $\phi$ is the characteristic function of some random variable $R$, then the random variables $R_n$ converge to $R$ in distribution. In our case, Propositions~\ref{mean} and \ref{var} along with Chebyshev's Theorem ensure that for any $\varepsilon>0$, almost all of the characteristic functions $\widehat{\nu_{z,n}}(t)$ are within $\varepsilon$ of $\widehat{\nu}(t)$. Thus we can take a subset of $z\in I_n$ where the individual gap measure of each $z$ converge to the average measure as $n$ tends to infinity and almost all $z\in I_n$ are chosen. \end{proof}

%%%%%%%%%%%%%%%%%%%%%%%%%%%%%%%%%%%%%%%%%%%%%%%%%%%%%%%%%%%%%%%%%%%%%%%%%%%%%%%%%%%%%%%%%%%%%%%%%%%%%%%%%%%%%%%%%%%%%%%%%%%%%%%%%%%%%%%%%%%%%%%%%%%%
%%%%%%%%%%%%%%%%%%%%%%%%%%%%%%%%%%%%%%%%%%%%%%%%%%%%%%%%%%%%%%%%%%%%%%%%%%%%%%%%%%%%%%%%%%%%%%%%%%%%%%%%%%%%%%%%%%%%%%%%%%%%%%%%%%%%%%%%%%%%%%%%%%%%
%%%%%%%%%%%%%%%%%%%%%%%%%%%%%%%%%%%%%%%%%%%%%%%%%%%%%%%%%%%%%%%%%%%%%%%%%%%%%%%%%%%%%%%%%%%%%%%%%%%%%%%%%%%%%%%%%%%%%%%%%%%%%%%%%%%%%%%%%%%%%%%%%%%%
%%%%%%%%%%%%%%%%%%%%%%%%%%%%%%%%%%%%%%%%%%%%%%%%%%%%%%%%%%%%%%%%%%%%%%%%%%%%%%%%%%%%%%%%%%%%%%%%%%%%%%%%%%%%%%%%%%%%%%%%%%%%%%%%%%%%%%%%%%%%%%%%%%%%

\ \\

\end{document}